\documentclass[12pt]{amsart}
\usepackage{amscd}

\usepackage{amssymb}

\newcommand{\bT}{{\mathbb T}}

\newcommand{\bZ}{{\mathbb Z}}
\newcommand{\bR}{{\mathbb R}}

\newcommand{\bA}{{\mathbb A}}

\newcommand{\bC}{{\mathbb C}}

\newcommand{\la}{{\langle}}
\newcommand{\ra}{{\rangle}}

\newtheorem{thm}{Theorem}[section]
\newtheorem{lemma}[thm]{Lemma}
\newtheorem{cor}[thm]{Corollary}

\numberwithin{equation}{section}

\begin{document}

\title[]{Definitions and examples of  algebraic Morava  K-theories}

 \author{Nobuaki Yagita}

\address{[N. Yagita]
Department of Mathematics, Faculty of Education, Ibaraki University,
Mito, Ibaraki, Japan}


\email{nobuaki.yagita.math@vc.ibaraki.ac.jp}



\begin{abstract}
Algebraic Morava K-theories are defined by  Sechin, Vishik, and others [Vi], [Ge-La-Pe-Se] as quotients of algebraic cobordism theories, which are oriented theories but not cohomology theories.  On the other hand the author had
  defined the Morava theories [Ya7] as (the double degree) generalized cohomology theories. 
We compare these Morava-K-theories.
\end{abstract}

\maketitle

\section{Introduction}

In the $\bA^1$-stable homotopy category,
over a field $k$ with $ch(k)=0$,  Voevodsky
[Vo1] defined the space (spectrum) $MGL(k)$ which defines the algebraic
(motivic) cobordism
\[ MGL^{*,*'}(X)\cong Mor( X, \ \bT^{*,*'}\wedge  MGL(k))
\]   where $Mor(-,-)$ is the group of morphisms in this category.

Soon later,  Levine-Morel [Le-Mo1,2] defined an another algebraic cobordism
$\Omega^*(X)$  for a smooth $X$ such that 
it is a universal oriented (for the formal groups laws).
In particular, Levine shows [Le]
\[\Omega^*(X)\cong MGL^{2*,*}(X)\quad for\ smooth \ X.\]
We will extend the above formula for the Morava $K$-theory, and compute
their examples. 

We first recall the Brown-Peterson cohomology (for details see $\S2$  below or [No]) such that \
\[ \Omega^*(pt.)_{(p)}\supset BP^*=\bZ_{(p)}[v_1,v_2,...] 
\quad |v_i|=-2p^i+2.\]
(For a topological cohomology theory $A^*(X)$, the notation $A^*$ means  its coefficients $A^*(pt)$, 
 but  for the algebraic case it means $A^{2*,*}(pt)$, here.)

Next recall that  
  the (topological) Morava $K$-theory
$K(n)^*(X)$ is a generalized cohomology theory
with the coefficient $K(n)^*=\bZ/p[v_n,v_n^{-1}]$. 
 In $\S 2$, we can construct [Ya7] the algebraic 
Morava K-theory $AK(n)^{*.*'}(X)$ with
$deg(v_n)=(-2p^n+2,-p^n+1)$.

We also recall the (topological) cobordism theory  $P(n)^*(X)$ in $\S 3$ and its algebraic version
$AP(n)^{*,*'}(X)$ such that
$ P(n)^*=\bZ/p[v_n,v_{n+1},...]$.
We define  
\[ AK(n)^{*,*'}(X) =  K(n)^*\otimes _{BP^*}AP(n)^{*,*'}(X).
\]

Let us write the ideal  $I_n=(p,v_1,...,v_{n-1})\subset BP^*$ so that $P(n)^*\cong BP^*/I_n$. 
Then we have (in $\S 3$)
\begin{lemma} 
For a smooth $X$, we have
\[ \Omega^*(X)/I_n
\cong \Omega(X)^*\otimes _{BP^*}P(n)^{*}\cong AP(n)^{2*,*}(X), \]
\[\Omega^*(X)\otimes _{BP^*}K(n)^*\cong
AK(n)^{2*,*}(X).\]
\end{lemma}

  On the other side, Sechin , Vishik and others [Vl], [Ge-La-Pe-Se]  define the (algebraic) $K(n)^*(X)$ by directly \
 $\Omega^*(X)\otimes _{BP^*}K(n)^*$
(write it here  by $OK(n)^*(X)$). 

To compute $\Omega^*(X),$ or $ \Omega(X)/I_n$ directly  is the rather difficult. In this paper, we first compute $ABP^{*,*'}(X),
AP(n)^{*.*'}(X)$ by using the Atiyah-Hirzebruch (type) spectral sequences.
Next we compute $OP(n)^*(X)$ and the last $OK(n)^*(X)$.

As examples for  Lemma 1.1, we consider the some cases $X=BG$ classifying spaces $BG$ of algebraic groups
in $\S 5$.
In $\S 6$, we consider 
the cases $X=G$ 
algebraic groups themselves.
In   $\S 7$, we see 
\begin{thm} Let $G$ be an algebraic group,
corresponding a simply connected Lie group
which has a $p$-torsion.
Then the algebraic group $G$ is not homotopy nilpotent in $\bA^1$-homotopy category.
\end{thm}
 In  $\S 8$,  we compute 
$AP(n)^{*,*'}(\tilde \chi_V)$  when $X$ is the reduced Ceck complex of the norm variety $V$.
In $\S 9$, we study the $K(n)$-theory of the norm variety over $k=\bR$.

\section{algebraic $BP$-theories}

At first, we recall 
the algebraic $MU^*$-theory $AMU^{*,*'}(X)$ 
from [Ya7].

For a topological space  $X$, recall that $MU^*(X)$ is the complex 
cobordism theory defined 
in the usual (topological) spaces and 
\[MU^*=MU^*(pt.)\cong \bZ[x_1,x_2,...]\quad |x_i|=-2i.\]
Here each $x_i$ is represented by a sum of hypersurfaces 
of $dim(x_i)=2i$ in some product of complex 
projective spaces ([Ha],[Ra]).
Let $MGL^{*,*'}(-)$ be the motivic cobordism theory defined by
Voevodsky [Vo1]. 
Let us write by $AMU$ the spectrum $MGL_{(p)}$ 
in the stable $\bA^1$-homotopy category representing this motivic cobordism theory (localized at $p$), i.e., 
\[ MGL^{*,*'}(-)_{(p)}=AMU^{*,*'}(-).\]
Here note that $AMU^{2*,*}(pt.)\cong MU_{(p)}^{2*}$.
It  is not  isomorphic to $AMU^{*,*'}(pt)$  in general,
 while $AMU^{*.*'}(X)$ is an $MU_{(p)}^*$-algebra.

Given a regular sequence $S_n=(s_1,...,s_n)$ with $s_i\in MU^*_{(p)}$,
we can inductively construct the $AMU$-module spectrum by
the cofibering of spectra ([Ya7,9], [Ra])
\[(2.1)\quad \bT^{-1/2|s_i|}\wedge AMU(S_{i-1}) \stackrel{\times 
s_i}{\longrightarrow}
AMU(S_{i-1})\to AMU(S_i)\]
where $\bT=\bA/(\bA-\{0\})$ is the Tate object.
For the realization map $t_{\bC}$ induced from $k\subset \bC$, it is also immediate that $t_{\bC}(AMU(S_n))\cong MU(S_n)$ with
$MU(S_n)^*=MU^*/(S_n).$ Therefore $t_{\bC}$ induces the natural map
\[ t_{\bC} \ :\  AMU(S)^{*,*'}(X)\to MU(S)^*(X). \]

Recall that the Brown-Peterson cohomology theory
$BP^*(-)$ with the coefficient  $BP^*\cong \bZ_{(p)}[v_1,v_2...]$ by
 identifying $v_i=x_{p^i-1}$. (So $|v_i|=-2(p^i-1)$.)
We can construct spectra (in the stable $\bA^1$-homotopy category)
\[ABP=AMU(x_i|i\not =p^j-1)\]
such  that $t_{\bC}(ABP)\cong BP$.
 For $S=(v_{i_1},...,v_{i_n})$, let us write
\[ABP(S)=AMU(S\cup\{x_i|i\not =p^j-1\})\]
so that $t_{\bC}(ABP(S))=BP(S)$ with $BP(S)^*=BP^*/(S)$.

In particular, let
$AH\bZ=ABP(v_1,v_2,...)$
so that $AH\bZ^{2*,*}(pt.)\cong \bZ_{(p)}$. 
 In the $\bA^1$-stable homotopy category,  
  Hopkins-Morel  showed that
\[ AH\bZ\cong H_{\bZ}\quad , i.e., \quad AH\bZ^{*,*'}(X)\cong H^{*,*'}(X,\bZ_{(p)})\]  
the (usual) motivic cohomology.  Using this result, we can construct the motivic Atiyah-Hirzebruch (type) spectral sequence. 
  \begin{thm} ([Ya7,9])   Let $Ah=ABP(S)$ for 
  $S=(v_{i_1},v_{i_2},...)$, and  
recall  
\[ h^{2*''}=BP^{2*''}/(S)\cong Ah^{2*'',*''}(pt).\]
Then there is AHss (the Atiyah-Hirzebruch
spectral sequence)
\[E(Ah)_2^{(*,*',2*'')}=H^{*,*'}(X;h^{2*''})\Longrightarrow Ah^{*+2*'',*'+*''}(X)\]
with the differential \quad 
$d_{2r+1}:E_{2r+1}^{(*,*',2*'')} \to 
E_{2r+1}^{(*+2r+1,*'-r,2*''-2r)}$.
\end{thm}

Note that $E_2^{*,*',2*''}$ 
  $\cong  H^{*,*'}(X.Ah^{2*'',*''}(pt)).$
The coefficient part is the only  difference place  from the usual AHss.

Note that  the cohomology $H^{m,n}(X,h^{2n'})$ here is the usual
motivic cohomology {with (constant) coefficients in} the abelian group $h^{2n'}$.   

{\bf Remark.}
We do $not$ asume the existence of the natural map $AE_r \to E_r$ of spectral sequences for algebraic to topological.
(while there is a map $Ah^{*,*'}(X)\to h^*(X)$)).

Here we recall some important properties
of the motivic cohomology. When $X$ is smooth, we
know the $(p-localized)$ Chow ring is  
\[(2.2) \ \  CH^*(X)\cong H^{2*,*}(X;\bZ_{(p)}),\quad \]
\[ (2.3)\ \ H^{*,*'}(X;\bZ_{(p)})\cong 0\ for \  2*'<*.\]
Hence  if $X$ is smooth, then
$E_r^{m,n,2n'}\cong 0$ for $m>2n.$

 Let $S\subset R=(v_{j_1},...)$.  Then
the  induced map $ABP(S)\to ABP(R)$ of spectra induces the 
$BP^*$-module map of AHss
:  $E(ABP(S))_r^{*,*',*''} \to E(ABP(R))_r^{*,*',*''}.$
In general, $ABP(S)^{*,*'}(X)\not \cong ABP^{*,*'}(X)/(S).$ 
However, from the above maps and dimensional reason 
(2.3)
with the differential (see also $\S 4$  below)
\[ d_{2r+1}:E_{2r+1}^{(2*,*,0)} \to 
E_{2r+1}^{(2*+2r+1,*-r,-2r)}=0, \]
 we see that all elements in 
$ E_2^{2*,*,0}
\cong H^{2*,*}(X)\otimes 1\cong 
 CH^*(X)$
 are infinite  cycles $d_r=0$ for all $r>0$.  
Hence, for a smooth $X$ (taking $S=\emptyset, R=S$) we have
the surjection and isomorphisms 
\[ (2.4)\quad BP^{2*}/(S)\otimes  CH^*(X) \twoheadrightarrow ABP(S)^{2*,*}(X)
\]
\[(2.5)\quad ABP(S)^{2*,*}(X)\otimes _{BP^*}\bZ_{(p)}\cong H^{2*,*}(X)\cong CH^*(X).\]

In this paper,  
a $connective $  $oriented $  theory $h^{2*}(X)$
means $ABP(S)^{2*,*}(X)$ as above.
We mainly consider 
the connective  oriented theory $ABP^{2*,*}(X)$.
Hereafter, we write it simply 
\[ \Omega^{*} (X)=ABP^{2*,*}(X)
\cong  MGL^{2*,*}(X)\otimes_{MU^*}BP^*.\]
Hence from (2.5), \ $\Omega^*(X)\otimes _{\Omega^*}\bZ_{(p)}\cong
CH^*(X)$ for a smooth $X$.


\section{algebraic Morava $K$ theories}

 By the arguments in previous section we can 
define the algebraic versions
$Ak(n)^{*.*'}(X),\  AK(n)^{*,*'}(X)$ and $ AP(n)^{*,*'}(X)$.

Recall the invariant ideal
$ I_n=(p,v_1,....,v_{n-1})\subset BP^*$.
Define  $P(n)=BP(I_n)$ so that
$P(n)^*=BP^*/(I_n)\cong \bZ/p[v_n,v_{n+1},...].$
(Since $I_n$ is invariant ideal, $P(n)^*(X)$ has the
Landweber-Novikov operations, but $K(n)^*(X)$ below  does 
not.)
 
The usual Morava $K$-theory  is defined from
the following (Conner-Floyd type) formula
\[(3.1)\quad K(n)^*(X)=K(n)^*\otimes_{P(n)^*}P(n)^*(X).\]
so that $K(n)^*=\bZ/p[v_n,v_n^{-1}]$.
(We also know  $K(n)^*(X)\cong k(n)^*(X)[v_n^{-1}]$, 
but most cases,  $k(n)^*(X)$ seems not easier 
 to compute than $P(n)^*(X)$.)
Remark that  for general, 
\[ K(n)^*(X)\not \cong K(n)^*\otimes_{ P(s)^*}{P(s)^*}(X)\quad for \  s\not =n.\]

We will study the algebraic version of these theorems.

Levine-Morel  define algebraic cobordism $\Omega^*(X)$
as the universal theory of the (some oriented) theories having the 
formal group laws.  In particular, Levine (essentially) shows
\[\Omega^*(X)\cong ABP^{2*.*}(X) \quad when \ \ X\ smooth.\]

Moreover, Sechin, Vishik and others define the oriented theories
\[ (3,2)\quad  OBP(S)^*(X)=\Omega^*(X)/(S),\]
\[ e.g. \ OP(n)^*(X)=\Omega^*(X)/I_n, \quad
OK(n)^*(X)=\Omega^*(X)\otimes_{BP^*}K(n)^*\]
 These  oriented theories are not represented by the motivic 
cohomology theories in general. But
we have the following lemma.

\begin{lemma}
Let $X$ be snooth.
 Then we have the ring isomorphism
\[ OP(n)^{2*}(X)=\Omega^*(X)/I_n
\cong AP(n)^{2*,*}(X).\]
\end{lemma}
\begin{proof}  
From (2.1) in the preceding section, we consider  the exact sequence for the motivic  cohomology
\[  AP(n)^{2*'-1,*'}(X)\stackrel{\delta}{\to}  AP(n-1)^{2*',*'}(X)\stackrel{v_{n-1}}{\to}
     AP(n-1)^{2*,*}(X) \to\] \[
AP(n)^{2*,*}(X)
\stackrel{\delta}{\to}AP(n-1)^{2*+1,*}(X)=0 \]
The last term in the above sequence follows from the dimensional reason (2.3).
  By induction,
we have
\[ AP(n)^{2*,*}(X)\cong AP(n-1)^{2*,*}(X)/(v_{n-1})\]
\[\cong OP(n-1)^{*}(X)/(v_{n-1})=
OP(n)^{*}(X). \]
\end{proof}

{\bf Remark.}  It seems unknown that $ABP(S)^{*,*'}(X)$ has a good ring structure for $S\not=\emptyset$. However $AP(n)^{*,*'}(X))$
has when $k=\bC$ (Theorem 7.4 in [Ya7]).

For other regular sequence
$S$ in  $\S 2$ in (2.1), by the same arguments, we have
\begin{lemma}
Let $X$ be smooth.
 Then we have a 
$BP^*$-module  isomorrphism
\[ OBP(S)^{2*}(X)=\Omega^*(X)/(S)
\cong ABP(S)^{2*,*}(X).\]
\end{lemma}

Next we recall the $truncated$ $BP$-theory.
Suppose that we fix generators $v_1,v_2,...$, (and some finiteness conditions,)
let us write
$ J_{n+1}=(v_{n+1},v_{n+2},...)\subset BP^* $
for the fixed generators (note $J_n$ is not
an invariant (under the Landweber-Novikov operations) ideal.
Define the truncated $BP$-theory
\[ BP\la n\ra^*(X)=BP(J_{n+1})^*(X) \]
so that 
$BP\la n\ra ^*\cong BP^*/J_{n*1}\cong 
\bZ_{(p)}[v_1,...,v_{n}] .$
Similarly, we can define
$ABP\la n\ra ^{*,*'}(X)$.  Moreover we have
from the preceding lemma 
\begin{lemma}
Let $X$ be snooth.
 Then we have $BP^*$-module   isomorphisms
\[ OBP\la n \ra ^{2*}(X)=\Omega^*(X)/J_{n+1}
\cong ABP\la n\ra^{2*,*}(X),\]
\[ Ok(n)^{2*}(X)=\Omega^*(X)/(I_n,J_{n+1})
\cong Ak(n)^{2*,*}(X).\]
\end{lemma}

{\bf Remark.}
  
The truncated theory 
$ABP\la n\ra ^{*,*'}(X)$
 is decided by the choice of 
%
$v_i\in J_{n+1}$ but not
decided  by $n$. Hence when we consider
$k(n)^*(X)$, we fix some typical generators of $BP^*$.

\section{Milnor operations and weight degree }

Voevodsky  showed  that there exists the Milnor operation $Q_i$
in the motivic cohomology $H^{*,*'}(X;\bZ/p)$ \ ([Vo1,3]).  For an object $\chi$ in $\bA^1$ homotopy category, we have 
\[ Q_i: H^{*,*'}(\chi;\bZ/p)\to H^{*+2p^i-1,*'+p^i-1}(\chi;\bZ/p)\]
which is compatible with the usual Milnor operation
on $H^*(t_{\bC}(\chi);\bZ/p)$ for the realization map $t_{\bC}$.
(Here the topological operation is defined $Q_0=\beta$
Bockstein operation and $Q_{i+1}=Q_i{P^{p^i}}-P^{p^{i}}Q_i$.)
This operation $Q_i$ can be extended on $H^{*,*'}(M;\bZ/p)$
for a motive $M$ in $M(X)$ (Lemma 7.1 in  \cite{YaB}).

The relation $Q_n$ to the exact sequence (2.1) for $k(n)^*(X)$ is giving
\[(4.1)\quad k(n)^*(X)\stackrel{v_n}{\to} 
k(n)^*(X)\stackrel{r}{\to} 
H^*(X;\bZ/p)\stackrel{\delta}{\to}
k(n)^*(X)\to... \]
Here we can see  $Q_n=r\delta$.

For $0\not =x\in H^{*,*'}(X;\bZ/p)$ (or cohomology operation),
let us write $*=|x|, w(x)=2*'-*, d(x)=*-*'$ so that $0\le d(x)\le dim(X)$
and $w(x)\ge 0$ when $X$ is smooth.  We also note
\[ |\tau|=0,\  w(\tau)=2,\quad |Q_i|=2p^i-1,\  w(Q_i)=-1\]
here $0\not =\tau\in H^{0,1}(Spec k);\bZ/p)$.

For the differential  $d_r$ in AHss, we have  
$w(d_r)=-1$.  Then the arguments for (2.4-2.5) are immediate.
In fact some differencial $d_r$ is represented as $v_n\otimes Q_n$.

We also note that (for $h^*(-)$ in $\S 2
$) 
\[ (4.2)\quad Ah^{2*.*}(X)\cong Ah^{*',*''}(X)/\{
	x\in  Ah^{*',*''}(X)|w(x)\ge1\} .\] 
.

\section{Examples ;  classifying spaces $X=BG$.}

We give here the most simple but a  non-trivial example, ([Ya5,8])
for Lemma 3.1.
 Let  $k=\bar k\subset\bC$, $X=BG$ and 
$G=\bZ/p $.  Hence
\[ H^{*,*'}(Spec( k);\bZ/p)\cong \bZ/p[\tau ].\]
 
At first, we recall the topological cases. Then (for $p$ odd)
\[ H^*(BG)\cong \bZ[y]/(py),\quad 
H^*(BG;\bZ/p)\cong \bZ/p[y]\otimes \Lambda(x),\]
where $|x|=1,\ |y|=2$ and $x^2=0$.
(For $p=2$, $x^2=y$.)

By the AHss, we have 
\[  E_2^{*,*'}\cong H^*(BG)\otimes BP^{*'}\cong BP^*[y]/(py) 
\Longrightarrow BP^*(X).\]
The $E_2$-term 
is generated by even degrees, and hence generates the $E_{\infty}$-term,
That is (the graded ring) $grBP^*(BG)\cong BP^*[y]/(py)$. 

 More precisely, we can see
\[
BP^*(BG)\cong BP^*[y]/([p](y)) \quad |y|=2. \]
(For easy of notations we simply write by $BP^*[y]$ the formal power series   $BP^*[[y]]$ in this paper.)
Here $[p](y)$ the $p$-th product of 
$BP^*$-formal group
law, in particular
\[[p](y)=py+v_1y^p+v_2y^{p^2}+... \ \ mod(I_{\infty}^2).\]	
Since $y=c_1$ the first Chern class, we have the surjection
\[ \Omega^*(X) \twoheadrightarrow
BP^*(BG)\cong BP^*[y]/([p]y).\]

 Next we consider the case $P(s)$ for $s\ge 1$.  The AHss is written
\[ E_2^{*,*'}\cong H^*(BG;\bZ/p)\otimes  P(s)^{*'}\Longrightarrow P(s)^*(BG).\]

It is known the first non-zero differencial is
\[ d_{2p^s-1}(x)=v_sQ_s(x)=v_sy^{p^s}.\]
Hence  we see \ 
\[grP(s)^*(X)\cong  E_{2p^s}\cong P(s)^*[y]/(v_sy^{p^s}). \]
In fact, the right hand side  
 is generated by even degree elements, and it is 
isomorphic to the infinite term.
In particular, $P(n)^{odd}(X)=0$.

The motivic   $ (k=\bar k) $  cohomology is written as
\[ H^{*,*'}(BG:\bZ/p)\cong H^*(BG;\bZ/p)\otimes 
\bZ/p[\tau],\quad \tau\in H^{0,1}(pt.;\bZ/p).\]
Here for $p=2$, $x^2=y\tau$ [Vo1-2].
 We can get the same result for the algebraic case
$AP(n)^{odd,*'}(BG)=0$. We consider the exact sequence
\[0= AP(1)^{2*-1,*}(X)\to ABP^{2*,*}(X)\stackrel{p}{\to} ABP^{2*,*}(X) \] \[ \to
AP(1)^{2*,*}(X)\to ABP^{2*+1,*}(X)=0 .\]

From the result of $AP(1)^{*,*'}(X)$ and the fact 
$\Omega^*(X)$ has no $p$-divisible element,  we can see
$\Omega^*(X)/p\cong AP(1)^{2*.*}(X).$
Thus we have 
\[\Omega^*(X)\cong BP^*(X)\cong BP^*[y]/([p](y)).\]

Hence we can apply
the main lemma.
\begin{lemma}
When $X=BG,\ G=\bZ/p$, we have
\[ OP(n)^*(X)=\Omega^* /I_n \cong AP(n)^{2*,*}(X).\]
\end{lemma}

Thus,  we can write  \[ AP(n)^{2*,*}(BG)\cong \Omega^*(BG)/I_n\cong
BP^*[y]/([p](y),I_n).\]
\[\cong \bZ/p[v_n,v_{n+1},...][y]/(v_ny^{p^n}+v_{n+1}y^{p^{n+1}}+...) \] 
\[
\cong P(n)^{2*}[y]/(v_ny^{p^n}+v_{n+1}y^{p^{n+1}}+...),\]
which has no $v_n$-torsion.
\begin{cor}        Let X=$BG$ for $G=\bZ/p$. Then
\[ AK(n)^{2*,*}(BG)\cong
 \Omega^*(BG)\otimes _{BP^*}
K(n)^* \] \[ \cong
K(n)^*[y]/(v_ny^{p^n})\cong K(n)^*[y]/(y^{p^n}).\]
\end{cor}
Hence the corresponding formal group law is the Honda group law (identifying $v_n=1$).

The similar  (but more weak) facts hold for other groups. 
 (Theorem 12.1, Proposition 12.4 in [Ya8])
\begin{lemma}  
Let $X=BG$ and $G$ be  $\bZ/p$, 
the orthogonal groups 
$ O_m$, $SO_m$
(or their products).
Then we have $AP(1)^{odd.*'}(X)=0$ and
\[ AP(1)^{*'',*'}(X)\cong AP(1)^{2*,*}(X)\otimes\bZ/p[\tau]\]
\[where \quad AP(1)^{2*,*}(X)\cong P(1)^{2*}(X)\cong BP^{2*}(X)/p 
.\]
\end{lemma}

{\bf Example.}
We consider the case $X=BG$ for $G=SO_3$ and $p=2$.  The cohomology is (given by the Stiefel-Whitney classes $w_2,w_3$
and Chern classes $c_2,c_3$) 
\[ H^*(X;\bZ/2)\cong\bZ/2[w_2,w_3].
\quad with \ \ Q_0w_2=w_3.\]
\[ Hence \ \ 
 H^*(X)\cong (\bZ\{1\}\oplus \bZ/2[c_3]^+)\otimes \bZ[c_2]\quad with \  c_i=w_i^2.\]
By using $Q_1(w_3)=w_3^2=c_3$  and the AHss, we have
\[  grBP^*(X)\cong (BP^*\{1\}\oplus 
 BP^* [c_3]^+/(2,v_1))\otimes \bZ[c_2] \]
Here note that in $BP^*(X)$, the element
$v_1c_3\not=0$ (it is $ v_2c_2^2+...$, this fact is shown considering the restriction
to $BP^*(B\bZ/2)$).

Hence we see 
$BP^*(X)/2\cong P(1)^*(X)$. 
(For  the details of the AHss converging $P(1)^*(X)$,
 see $\S 12$ in [Ya8].)

\section{Lie groups}

Let $k=\bar k\subset \bC$.
In this section, we study 
for $X=G$ ; Lie groups themselves.
For $G$ exceptional groups and $p\not =2$ are studied in [Ya1,2,3].

Hereafter 
let $X=G=SO_7$,
Then
\[ H^*(G;\bZ/2)\cong \Lambda(x_3,x_5,y_6),
\quad x_3^2=y_6 \]
\[ \quad H^*(G)\cong \Lambda(x_3)\otimes
(\bZ\{1,x_5y_6\}\oplus \bZ/2\{y_6\}).\]
Here the suffix $i$ of $x_i$, $y_6$  means its degree and  $x_5y_6$ is an element in $H^*(G)$
such that its image in $H^*(G;\bZ/2)$ is the same name.

We consider AHss
\[ E_2\cong H^*(G)\otimes BP^*
\Longrightarrow BP^*(G).\]
The differencial is
given by $d_3(x_3)=v_1Q_1x_3=v_1y_6. $
Hence infinitive term is writen as 
\[E_{\infty}\cong BP^*A\oplus
BP^*B/2\oplus BP^*C/(2,v_1)\]
\[with \quad A=\bZ\{1,2x_3,x_5y_6,x_3x_5x_6\},\quad 
B=\bZ/2\{x_3x_6\}, \quad C=\bZ/2
\{y_6\}.\]
Here $2y_3$ and $x_iy_6$ are elements in $E_{\infty}$ induced from the same named elements in $E_2$.

Then we can compute more  exactly [Ya1,2]
\[   BP^*(G) \ \  \cong \quad  \quad    BP^*\{1,2x_3,x_3x_5x_6\}\qquad  \qquad
\qquad \qquad \qquad \qquad  \]
\[\qquad \qquad \qquad  \qquad   \oplus \  
BP^*\{x_3y_6,x_5y_6]/(2x_3y_6-v_1x_5y_6) \]
\[
\oplus BP^*/(2,v_1)\{y_6\}.\]

Similarly, we can compute (for example
using $P(2)^*(X)\cong P(2)^*\otimes H^*(X;\bZ/2)$)
\[ P(1)^{*}(G)\cong( P(1)^*\{1,x_3y_6\}
\oplus P(2)^*\{y_6\})\otimes \Lambda(x_5).\]
Indeed, take  $x_5\in P(1)^5(X)$ with 
$\delta(x_5)=y_6$. The element $2x_3\in BP^*(X)$ goes (by the reduction $u$) to $v_1x_5\in P(1)^*(X)$
i.e., $u(2x_3)+v_1x_5=0$, recall $Q_1x_3=Q_0x_5$.

By Kac [Ya6], we knew ($CH^*(X)/2=\bZ/2\{1,y_6\}$)
\[  \Omega ^*(G)/I_{\infty}\cong CH^*(G)/2\cong \bZ/2\{1\}\oplus  \bZ/2\{y_6\}. \]
Moreoer we know $v_1y_6=0$ from $Q_1x_3=y_6$.  Hence we see
$w(y_6)=0$ but $ w(x_ix_j)\ge 1$ for
the other generators..
Thus we get
\[ ABP^{2*,*}(G)\cong \Omega^*(G)\cong
BP^*\{1\}\oplus
P(2)^*\{y_6\}.\]

Since $k=\bar k$, we see  $H_{et}^*(Spec(k);\bZ/2)\cong \bZ/2[\tau]$. Hence each $\tau^i$ is a permanent cycle in the above AHss, and 
we get 
\begin{thm} Let $G=SO_7$
 Then
\[AP(1)^{*,*'}(G) \cong \Omega^*(G)\otimes \bZ/2[\tau],\]
\[where\quad \Omega^*(G)/2\cong P(1)^*\{1\}\oplus P(2)^*\{y_6\} .\]
\end{thm}
\begin{cor}  Let $G=SO_7$.  Then
\[ \ AK(s)^{2*,*}(G)\cong \begin{cases}
 K(1)^*\{1\} \quad for \ s=1 \\
 K(2)^*\{1,y_6\} \quad for\ \ s\ge 2\end{cases}\]
\end{cor}

Quite recently, Gelthauser- Lavrenov-
 Petrov-Sechin [Ge-La-Pe-Se]
study the Morava  $K$-theories for all $SO_m$.

\section{Homotopy Nilpotency}

Let $X=G$ be a simply connected Lie group.  For ease of arguments, let us assume $p$ odd in this section.
By Borel, its $mod(p)$ cohomology  is a tensor product of truncated polynomial and an exterior algebra.

 Moreover when it has $p$-torsion. there is an surjective map [Ya4]
\[ (7.1)\quad H^{*}(X;\bZ/p)\twoheadrightarrow
A=\bZ/p[y]/(y^p)\otimes \Lambda(x,x')\]
\[where \quad |x|=3,\ \ P^1(x)=x',\ \ Q_1x=Q_0x'=y, \]
(Hence $|y|=2p+2$, $|x'|=2p+1$) 

 We also see, by dimensional reason, that for $X$ in (7.1) \[ (7.2)
\quad   K(2)^*(X)\twoheadrightarrow
K(2)^*\otimes A\]
We consider the homology theory (the $\bZ/p$-dual of the 
cohomology theory).  There is the injection (but $not$ a $ring$ map)
\[ K(2)_*\otimes B \subset K(2)_*(X) \
 where \quad B= \bZ/p[y]/(y^p)\otimes \Lambda(z,z'), \]
with $y$ ( resp.  $z.z'$ ) is the dual of $y$ (resp. $x,x'$)

We want to study the Pontryagin ring structure of $K(2)_*(X)$, i.e., the product is induced from that of the Lie group $G$.
Let us write the adjoint 
\[ ad(y)(z)=[y,z]=yz-zy.\]
(Here the product is the Ponriyagin product.)
The definition of homotopy nilpotents  implies  as all Pontriyagin products are nilpotents.  In particular 
\[ ad^i(y)(z)=0\quad for \ large \ i.\]

\begin{lemma} 
([Ya4] Theorem 1.5)
 Let $G$ in the   group (7.1).  Then $G$ is not homotopy nilpotent. 
\[ ad^{p-1}(y)(z)=-v_2z\not =0  \quad \in \ K(2)_*(G).\]
\end{lemma}
We  know 
\[x\in H^{3,3}(G;\bZ/p)
\cong H_{et}^3(G;\bZ/p)
\cong
H^3(X(\bC);\bZ/p).\]
Hence $x$ (resp. $z$) exists in $AK(2)^{*.*'}(X)$ (resp.$A K(2)_{*.*'}(X))$
By the compatible of $t_{\bC}$ and $K2)^*$ we have 
that $X$ is not homotopy nilpotent.

\begin{lemma} Let $G$ be an algebraic group,
corresponding a simply connected Lie group
which has a $p$-torsion.
Then $G$ is not homotopy nilpotent in $\bA^1$-homotopy category.
\end{lemma}
 
{\bf Examples.} Let $G=F_4,p=3$ (or $G=E_8, p=5$).  Then we
have 
\[ K(2)_*(G)\cong K(2)_* [y]/(y^p)\otimes C\otimes C'  \]
\[where \quad \begin{cases}C=\Lambda(z,ad^1(y)(z),...ad^{p-2}(y)(z)) \\
C'=\Lambda(z',ad^1(y)(z'),...ad^{p-2}(y)(z')).
\end{cases}\]
\[\quad ad^{p-1}(y)(z)=-v_2z.\quad 
 ad^{p-1}(y)(z')=-v_2z'.\]

\section{$C\hat ec$k complex for
 $k\not =\bar k$}

Suppose $ch(k)=0$  but we do not  assume
$k=\bar k$ in this section. (We do not assume $X$ smooth.)
The following lemmas also hold for the algebraic theories.
Let us write $Q(n)=\Lambda(Q_0,...,Q_n)$ for the Milnor operation $Q_i$.

\begin{lemma}
Let $P(1)^*(X)\cong  P(n)^*\otimes B_n, $ for some $\bZ/p$-module $B_n$, (i.e.
it is an $I_n$-torsion and $P(n)^*$-free module ).
Then \[P(s)^*(X)\cong P(n)^*\otimes Q(s-1)\otimes B_s\quad for \ s\le n
\]
where $Q_1...Q_{n-1}B_n=B_1$,  $ (i.e.,\ B_s=Q_1^{-1}...Q_s^{-1}B_1)$.
Hence \[ \quad K(s)^*(X)\cong
\begin{cases} 0\quad s<n \\
  K(n)^*\otimes Q(n-1)\otimes B_n
\quad  s= n
\end{cases}
\]

\end{lemma}
\begin{proof}
Let $s< n$. 
By induction, we assume
\[P(s)^*(X)\cong P(n)^*\otimes Q(s-1)B_s.\]
 We consider the exact sequence (2.1)
\[   \begin{CD}
 P(s)^*(X)@>{v_s}>>  P(s)^*(X)@>{r}>>  P(s+1)^*(X)
@>{\delta}>>  P(s)^{*+1}(X)...
\end{CD} \]
Here we assume $P(s)^*(X)$ is a $P(n)^*$-module  by induction.
That is $v_s=0$.  Hence we have
\[ P(s+1)^*(X)\cong  \bZ/p\{r,\delta^{-1}\}P(s)^*(X)\]
\[ \cong  \bZ/p\{r,\delta^{-1}\}\otimes P(n)^*\otimes Q(s-1)B_s
\quad by \ induction\]
\[ \cong  \bZ/p\{1,Q_s^{-1}\}\otimes P(n)^*\otimes Q(s-1)B_s
\quad by \ Q_s=r\delta\]
\[\cong P(n)^*\otimes Q(s-1)\otimes\bZ/p\{ B_s, B_{s+1}\} \quad by\  B_s=Q_sB_{s+1}. \]
\[\cong P(n)^*\otimes Q(s)\otimes B_{s+1}. \]
Taking $s=n-1$, we get the result.  \end{proof}

\begin{cor}  By the same assumption as the above lemma.
we see that
\[H^*(X;\bZ/p)\cong Q(n-1) \otimes B_n\]
\end{cor}
\begin{proof}
We study the exact sequence in the proof of the
preceding lemma, but for $s\ge n+1$.  Since the map 
$v_s$ on $ P(s)^*(X)\cong P(s)^*B_n$ is  injective, we have
\[P(s+1)^*(X)\cong P(s+1)^*\otimes B_n.\]
We get the corollary
(for the  algebraic case  from the Hopkins  and Morel  theorem 
$ lim_{s\to \infty}AP(s)^{*,*'}(X)
\cong
 H^{*,*'}(X;\bZ/p).$)
\end{proof}
\begin{lemma}  Suppose that as a $Q(n-1)$-module 
\[H^*(X;\bZ/p)\cong Q(n-1)\otimes B_n. 
\]
Then we have $P(1)^*(X)\cong P(n)^*\otimes B_1$.
\end{lemma}
\begin{proof}
We consider AHss
\[ 
E_2\cong H^*(X;\bZ/p)\otimes P(1)^*
\cong Q(n-1)B_n\otimes P(1)^* \Longrightarrow P(1)^*(X).\]
By the naturality, the first non-zero differential is given
\[ d_{2p-1}(B_n)= v_1\otimes Q_1B_n.\]

Hence $Ker(d_{2p-1})$ is $ P(1)^*\otimes Q_1B_n$ from $(Q_1)^2=0$.  Moreover we see
\[ E_{2p}\cong P(1)^*\otimes Q_1B_n/(v_1Q_1B_n).\]

By induction and the assumption of degree
and $d_{2p^i-1}(x)=v_i\otimes Q_ix$, we can compute
\[ E_{\infty}\cong P(1)^*/(v_1,v_2 ,..., v_{n-1})\otimes Q_1Q_2...Q_{n-1})B_n\]
which is isomorphic to $P(n)^*\otimes B_1$.
(Note that $Q_1...Q_{n-1}B_n\in P(1)^*(X)$,)
\end{proof}

\begin{cor}
Let $P(1)^*(X)\cong P(n)^*\otimes B_n$.  Then
\[ K(n)^*(X)\cong K(n)^*\otimes B_n,\quad and\ \  
K(s)^*(X)\cong 0 \ \ s<n.\]
\end{cor}
\begin{proof}
 We have the corollary  from $K(s)^*\otimes P(n)^*_{BP^*}\otimes B_s=0$.
\end{proof}

	{\bf Example 1.} $V(n)$. 
 
In the (topological) stable homotopy category,  the space $X=V(n)$
is called the Smith-Toda space (spectrum) if
\[ P(1)^*(X)\cong P(n+1)^* \quad
(or\ H^*(X;\bZ/p)\cong Q(n)).\]

If $V(n)$ exists then we can define a ( non-zero) stable homotopy
element of a sphere
\[ S^0 \to V(n)\stackrel{v_n}{\to}V(n)\stackrel{qu.}
{\to}
S^N \quad some \ N\]
This is thought as an important problem in the stable homotopy groups of spheres.
However $V(n)$ does not exist in general.  For example
$V(3)$ exists if and only if $p\ge 7$.

{\bf Example 2 } ; norm variety for  $\bar k\not =k$.

Let $\chi_X$ be the $\hat Cech$ complex, which is defined
as $(\chi_X)^n=X^{n+1}$ (see details [Vo1,2,4]).  Let
$\tilde \chi_X$ be defined from the cofiber sequence 
\[   \tilde \chi_X\to \chi_X \to Spec(k)\]
in the stable $\bA^1$-homotopy category.
Voevodsky defined the motivic cohomology $H^{*,*'}(\chi;\bZ/p)$
for all objects $\chi$ in the stable $\bA^1$-homotopy category.

It is known that
the ideal 
\[ I(X)=\pi_*BP^*(X)\subset BP^*(pt.)=BP^*\quad for\ \pi: X\to pt.\]
is an invariant ideal (e.g. Lemma 5.3 in [Ya7]).

The following two lemmas are known.
\begin{lemma} [Vo1]
Let $p:X\to Spec(k)$ be the projection and $t_{\bC}(X)=v
\in  BP^*$.  Let $Ah=ABP(S)$ for some regular sequence $S$.
Then 
\[Ah^{*,*'}(\tilde \chi_X)\  is \ \ v-torsion.\]
\end{lemma}
\begin{proof}
Let us write $\chi _X$ by $\chi$ simply.
We consider the maps
\[ Ah^{*.*'}(\chi)\stackrel{p^*}{\to
}   Ah^{*,*'}(\chi\times X)
\stackrel{p_*}{\to}
Ah^{d,d'}(\chi)  \]
so that $p_*p^*(x)=vx$.
The lemma is proved from $Ah^{*.*'}(\tilde \chi)=0$ because
\[ Ah^{*,*'}(\chi\times X)\cong 
Ah^{*,*'}(\chi).\]

\end{proof}

\begin{lemma} (Lemma 6.4,6.5 in \cite{YaB})
If $I_{n}\subset I(X)$, then
$ABP^{*,*'}(\tilde \chi_X)$ is $I_n$-torsion,
and $H^{*,*'}(\tilde \chi_X;\bZ/p)$ is $Q(n-1)$-free.
\end{lemma}

Now we recall the norm variety.
Given   a pure symbol $a$ in the mod $p$  Milnor $K$-theory
$ K_{n}^M(k)/p$, by Rost,  we can
construct the norm variety $V_a$ such that 
\[ \pi_*([V_a])=v_{n-1},\quad a|_{k(V_a)}=0\in K_{n}^M(k(V_a))/p\]
where $[V_a]=1\in \Omega^0(V_a)$,  $\pi:V_a\to pt.$ is the projection and $k(V_a)$ is the function field of $V_a$ over $k$.  
Note $I(V_a)=I_{n}$.

\begin{thm} (For $p=2$, [Ya4])
 Let $0\not =a=(a_0,...,a_n)\in K_{n}^M(k)/p$.  Then 
there is a $K_*^M(k)\otimes Q(n-1)$-modules isomorphism
\[H^{*,*'}(\tilde{\chi}_a;\bZ/p)\cong K_*^M(k)/(Ker(a))\otimes 
Q(n-1)\otimes \bZ/p[\xi_{a}]
\{a'\}\]
where $\xi_{a}=Q_{n-1}....Q_0(a')$ and  $deg(a')=(n,n-1)$.
\end{thm}

Thus from Lemma 8.3 and Lemma 3.1,  we see 
\begin{cor}
We have
\[ 
AK(s)^{*,*'}(\tilde{\chi}_a)\cong
\begin{cases}
K(n)^*\otimes 
H^{*,*'}(\tilde{\chi}_a;\bZ/p)
\quad s=n\\
0\quad 1\le s<n.
\end{cases}\] 
\end{cor}

///

{\bf Remark.}
For details of etale cohomologies of
$\chi _a$ and $M_a$ are studied 
in $\S 8,9$ in [Ya 10].
For an  example
\[ H^{*,*'}(M_a;\bZ/2)\cong H^{*,*'}(\chi_a;\bZ/p)\quad for *<2(p^{n-2}-1)/(p-1).\]

\section{ The motivic cohomology of quadrics over $\bR$ with coefficients $\bZ/2$}

Let $X$ be a smooth variety over the field
 $\bR$ of real numbers,
and we consider the cohomologies of $\bZ/2$ coefficients.
In this paper the $mod(2)$ \'etale cohomology means the 
motivic cohomology of the same first and the second degrees
$ H_{\acute{e}t}^*(X;\bZ/2)\cong H^{*,*}(X.\bZ/2)$.

It is well known ([Vo1], [Vo2])
\[ H_{\acute{e}t}^*(Spec(\bC);\bZ/2)\cong \bZ/2, \quad 
 H^{*,*'}(Spec(\bC);\bZ/2)\cong \bZ/2[\tau], \]
\[ H_{\acute{e}t}^*(Spec(\bR);\bZ/2)\cong \bZ/2[\rho], \quad 
 H^{*,*'}(Spec(\bR);\bZ/2)\cong \bZ/2[\tau, \rho] \]
where $0\not=\tau\in H^{0,1}(Spec(\bR);\bZ/2)\cong \bZ/2$ and
where 
\[\rho=-1\in \bR^*/(\bR^*)^2\cong K_1^M(\bR)/2\cong H_{\acute{e}t}^1(Spec(\bR);\bZ/2).\]

We recall the cycle map from the Chow ring to the \'etale cohomology
\[cl/2: CH^*(X)/2\to H^{2*}_{\acute{e}t}(X;\bZ/2).
\]
This map is also written as
$ H^{2*,*}(X;\bZ/2)\stackrel{\times \tau^*}{\to}
H^{2*,2*}(X;\bZ/2).$

Let $X=Q^d$ be an anisotropic quadric of dimension $2^n-1$
(i.e. the norm variety for $\rho^{n+1}\in K_{n+1}^M(\bR)/2$).  Then we have the Rost motive  $M\subset Q^d$ [Ro].    
 It is known ( the remark in  page 575 in [Ya2])
\[ H_{\acute{e}t}^{*}(M;\bZ/2)\cong \bZ/2[\rho]/(\rho^{2^{n+1}-1})
\cong \bZ/2\{1,\rho,\rho^2,...,\rho^{2^{n+1}-2}\}.\]

The 
Chow ring is also known [Ro]
\[CH^*(M
)/2\cong \bZ/2\{1,c_0,c_1....,c_{n-1}\}, 
\quad cl(c_i)=\rho^{2^{n+1}-2^{i+1}}.\]
The cycle map $cl/2$ is injective.
The elements $c_i$ is also written as 
\[ c_i=\rho^{2^{n+1}-2^{i+1} }
\tau^{-2^n+2^{i}}  
\quad  in \ CH^*(M)/2\subset H_{\acute{e}t}^{2*}(M:\bZ/2)[\tau^{-1}]. \]

The mod(2) motivic cohomology is known 
(Theorem 5.3 in \cite{Ya2}).

\begin{thm} (Theorem 5.3 in \cite{Ya2}, [Ya3])
  The cohomology 
$ H^{*,*'}(M_n;\bZ/2)$ is isomorphic to the  $\bZ/2[\rho,\tau]$-subalgebra of
\[ \bZ/2[\rho,\tau,\tau^{-1}]/(\rho^{2^{n+1}-1}) \]
generated by  \ \ $ a=\rho^{n+1} ,\  \ a'=a\tau^{-1},\  $  and \ elements in $  \Lambda (Q_0,...,Q_{n-1})\{a'\}.   $
\end{thm}
\begin{lemma}   We have 
$Q_0(\tau^{-1})=\rho \tau ^{-2}$.  Hence  $Q_0(a')=\rho a\tau^{-2}, $  while $Q_0(a)=0$.
Similarly, we see $Q_1(\tau^{-2})=\rho^3\tau^{-4}$, and 
$Q_1(a')=\rho^6\tau^{-3}$.
\end{lemma}
\begin{proof}
We see the first equation from
\[ 0=Q_0(1)=Q_0(\tau \tau^{-1})=\rho \tau^{-1}+\tau Q_0(\tau^{-1}).\]
\end{proof}


{\bf Example.}  $X=M_2$.
Recall that $\tau : H^{*,*'}(X;\bZ/2)
\to H^{*,*'+1}(X;\bZ/2)$ is injective.
The motivic cohomology is given as
(Example 5.15 in [Ya9])
\[ H^{*,*'}(X;\bZ/2)\cong 
\bZ/2[\tau]\{1,\rho,\rho^2, a'=\rho^3\tau^{-1}, \qquad \qquad \]
\[ \qquad  \qquad \qquad Q_0a'=\rho^4\tau^{-2},
    \rho Q_0a'=\rho^5\tau^{-2},
    Q_1a'=\rho^6\tau^{-3}\}.\] 

We consider AHss for $s=1$ or $2$
\[ E_2\cong H^{*,*'}(X;\bZ/2)\otimes P(s)^*\Longrightarrow AP(s)^{*,*'}(X). \] 

We first consider the case $s=2$.  The 
first  degree of the differential
are ( see Theorem 2.1)
\[ -|v_2|+1>6= first\ degree\ of H^{2*,*}(X:\bZ/2).\]
Hence all differentials 
(to $CH^*(X)/2$) are zero.

Next, we consider the case $s=1$
We note 
\[d_3(a')=Q_1a'=\rho^6\tau^{-3}.\]
Then we see
\begin{thm}  Let $X=M_2$.
\[ AP(s)^{2*,*}(X)\cong
\begin{cases}
P(1)^*\{1,Q_0a'\}\oplus P(2)^{*}\{{Q_1}a'\}\quad s=1 \\
P(s)^*\{1,Q_0a',Q_1a'\}\quad s\ge 2.
\end{cases} \]
\end{thm}

Recall $Q^3$ is the norm variety $M_2
\subset Q^3$. We have the decomposition of the motives
$M(Q^3)\cong M_2\oplus M_1\otimes  \bT$ (see $\S 8$ in [Ya9]).
 The Chow ring is written
\[ CH^*(Q^3)\cong \bZ_2\{1,h,h^2,h^3\}\oplus \bZ/2\{c\}.\]
Here $h\in CH^1(\bT)$, $ h^2\in CH^2(M_1)$,   $h^3=Q_1a'$ and  $c=Q_0a'$.  Hence 
\begin{cor}
For the norm variety $Q^3$ over $\bR$, we have 
\[AP(1)^{2*,*}(Q^3)\cong
P(1)^*\{1,h,h^2,Q_0(a')\}
\oplus 
P(2)^*\{h^3\}.\]
\[AK(1)^{2*,*}(Q^3)\cong
K(1)^*\{1,h,h^2,Q_0(a')\}
.\]
\end{cor}

Remark that (Lemma 6.2 in [Ya9]))
we know
$Q_1(\tau^2)=\rho^3$.
However this element does note related
elements of $w(x)=0$ (In fact $w(\tau)=2$).

\begin{thm}  Let $X=M_n$.
\[ AP(s)^{2*,*}(X)
\begin{cases} \supset 
P(n)^*\{Q_1...Q_{n-1}a'
\} \quad s=n-1
\\ \cong 
P(s)^* \otimes CH^*(X)^*\}\quad s\ge n
\end{cases} \]
\end{thm}

Recall $ a=\rho^{n+1}.
\quad a'=a\tau^{-1}.$ 
So $w(a')=n-1$.  In fact, we have  
\[ CH^*(X)/2\cong H^{2*,*}(X;\bZ/2) 
\cong 
\bZ/2\{Q_0...\hat Q_i...Q_{n-1}(a') |
0\le i\le n-1\}.\]


\begin{thebibliography}{R-W-Y}



\bibitem[Ge-La-Pe-Se]{Ge-La-Pe-Se}
N.Geldhauser, A.Lavrenov, V.Petrov, P.Sechin.
\newblock  Morava J-invariant
\newblock 
ArXiv. 2024.14099v1 [Math.KT]

\bibitem[Ha]{Ha}
M.Hazewinkel.
\newblock  Formal groups and applications.
\newblock \emph{Pure and applied Math. 78., Academic Press,Inc.}
(1978), 573pp.










\bibitem[Le]{Le}
M.~Levine.
\newblock  Comparison of cobordism theories.  
\newblock \emph{J. Algebra} \textbf{322} (2009), 3291-3317.

 

\bibitem[Le-Mo1]{Le-Mo1}
M.~Levine and F.~Morel.
\newblock  Cobordisme alg\'ebrique I.
\newblock \emph{C. R. Acad. Sci. Paris } \textbf{332} (2001), 723-728.

\bibitem[Le-Mo2]{Le-Mo2}
M.~Levine and F.~Morel.
\newblock  Cobordisme alg\'ebrique II.
\newblock \emph{C. R. Acad. Sci. Paris } \textbf{332} (2001), 815-820.


\bibitem[Mi]{Mi}  
J.Milnor. 
\newblock   On cobordism ring $\Omega_*$ and a complex analogue,I.
\newblock \emph{Amer. J. Math.} 
 \textbf{82} (1960), 505-521. 
 






\bibitem[Mo-Vo]{Mo-Vo}
F.~Morel and V.~Voevodsky. 
\newblock $\bA^1$-homotopy theory of schemes.
\newblock \emph{IHES Publ. Math.}  \textbf{90} (2001), 45-143.



\bibitem[No]{No} 
P.~Novikov
\newblock The methods of algebraic topology from the view point of cobordism 
theory.
\newblock \emph{Math. USSR. Izv. } \textbf{1} (1967), 827-913.


\bibitem[Ra]{Ra}
D.Ravenel.
\newblock Complex cobordism and stable homotopy groups of spheres. 
\newblock \emph{Pure and Applied Mathematics, 121. Academic Press}
(1986).





\bibitem[Qu]{Qu}  
D.~ Quillen. 
\newblock Elementary proofs of some results of cobordism theory using
Steenrod operations. 
\newblock \emph{Adv. Math.} \textbf{7} (1971), 29-56.


\bibitem[To1]{To1}  
B.~Totaro. 
\newblock Torsion algebraic cycles and complex cobordism.
\newblock \emph{J. Amer. Math. Soc.} \textbf{10} (1997), 467--493.


\bibitem[To2]{To2} 
B.~Totaro.  
\newblock The Chow ring of classifying spaces. 
\newblock \emph{Proc.of Symposia in Pure Math. "Algebraic K-theory" 
(1997:University of Washington,Seattle)} \textbf{67} (1999), 248-281.







\bibitem[Vi]{Vi}  
A.Vishik. 
\newblock Isotropic and  numerical equivalence for the Cow rings and Morava K-theories   
\newblock  \emph{Invent. Math.}
\textbf{237} (2024), 770-808.




\bibitem[Vo1]{Vo1} 
V.~Voevodsky. 
\newblock The Milnor conjecture.
\newblock \emph{www.math.uiuc.edu/K-theory/0176}
\textbf{} (1996).

\bibitem[Vo2]{Vo2}
V.~Voevodsky (Noted by Weibel).
\newblock Voevodsky's Seattle lectures : $K$-theory and motivic cohomology
\newblock \emph{Proc.of Symposia in Pure Math. "Algebraic K-theory" 
(1997:University of Washington,Seattle)} \textbf{67} (1999), 283-303.



\bibitem[Vo3]{Vo3} 
V.~Voevodsky. 
\newblock Motivic cohomology with $\bZ/2$ coefficient.
\newblock \emph{Publ. Math. IHES}
\textbf{98} (2003), 59-104.


\bibitem[Vo4]{Vo4} 
V.Voevodsky. 
\newblock Reduced power operations in motivic cohomology.
\newblock \emph{Publ.Math. IHES}
\textbf{98} (2003),1-57.





\bibitem[Vo6]{Vo6 } 
V.Voevodsky. 
\newblock On motivic cohomology with $\bZ/l$-coefficients.
\newblock \emph{Ann. Math.}
\textbf{174}  (2011) 401-438.


 
 
 
 

 




\bibitem[Ya1]{YaJ}
N.~Yagita.
\newblock  Brown-Peterson cohomology of exceptional Lie group 
\newblock \emph{J. Pure and Applied Algebra }
{\bf 17} (1980), 223-23223-226.

\bibitem[Ya2]{YaRo}
N.~Yagita.
\newblock  On the mod odd prime Brown-Peterson cohomology groups of exceptional Lie groups.
\newblock{J. Math. Soc. Japan}
 \textbf{34}  (1982) 294-305.




\bibitem[Ya2]{Ya3}
N.~Yagita.
\newblock  On relations between Brown-Peterson cohomology and the ordinary mod
$p$ cohomology theory.. 
\newblock \emph{Kodai Math. J.} \textbf{7} (1984), 273-285.

/////


\bibitem[Ya4]{Ya1}
N.~Yagita.
\newblock  Pontrojagin
rings of the Morava K-theory
for finite H-spaces.
newblock \emph{J. Math. Kyoto Univ.
c.} \textbf{36}  (1996), 447-452.




\bibitem[Ya5]{Ya3}
N.~Yagita.
\newblock  Examples for the mod p motivic cohomology of classifying spaces.
\newblock \emph{Trans. Amer. Math. Soc.} \textbf{335} (2003) 
4427-4450.


\bibitem[Ya6]{Ya1}
N.~Yagita.
\newblock  Algebraic cobordism of simply connected Lie groups.
\newblock \emph{Math.Proc.Camb.Phil.Soc.} \textbf{139} (2005), 243-260.



\bibitem[Ya7]{Ya2}
N.~Yagita.
\newblock  Applications of Atiyah-Hirzebruch spectral 
sequence for motivic cobordism.
\newblock \emph{Proc. London Math. Soc.}
\textbf{90} (2005) 783-816.

\bibitem[Ya8]{Ya3}
N.~Yagita.
\newblock  Coniveau filtrations of cohomology of groups.  
\newblock \emph{Proc. London Math. Soc.}
\textbf{101} (2010), 179-206.


\bibitem[Ya9]{Ya2}
N.~Yagita,
\newblock  Motivic cohomology of quadrics and the coniveau spectral sequence.
\newblock {J. K-theory}
\textbf{6} (2010) 547-589.


 
\bibitem[Ya10]{YaB}
N.~Yagita.
\newblock  Algebraic $BP$-theory and norm varieties.
\newblock \emph{Hokkaido Math. J.}
{\bf 41} (2012), 275-316.






\end{thebibliography}
\end{document}